\documentclass[preprint]{elsarticle}
\usepackage{amsmath,amssymb}

\journal{Applied Mathematics and Computation}

\begin{document}

\begin{frontmatter}
\title{Explicit formulas for determinantal representations of the Drazin inverse
solutions of some matrix and differential matrix equations. }
\author{Ivan Kyrchei}

\address{Pidstrygach Institute
for Applied Problems of Mechanics and Mathematics NAS of Ukraine,
str.Naukova 3b, Lviv, 79060, Ukraine,  kyrchei@lms.lviv.ua}

\begin{abstract} The Drazin inverse
solutions of the matrix equations ${\rm {\bf A}}{\rm {\bf X}} =
{\rm {\bf B}}$, ${\rm {\bf X}}{\rm {\bf A}} = {\rm {\bf B}}$ and
${\rm {\bf A}}{\rm {\bf X}}{\rm {\bf B}} ={\rm {\bf D}} $ are
considered in this paper. We use both the  determinantal representations of the Drazin inverse obtained earlier by the author and    in the paper. We get analogs of the Cramer rule for the Drazin inverse
solutions of these matrix equations and using their for determinantal representations  of solutions of some differential matrix equations, ${\bf X}'+ {\bf A}{\bf X}={\bf B}$ and ${\bf X}'+{\bf X}{\bf A}={\bf B}$, where the matrix ${\bf A}$ is singular.
\end{abstract}

\end{frontmatter}

\section{Introduction}
\newtheorem{definition}{Definition}[section]
\newtheorem{lemma}{Lemma}[section]
 \newtheorem{corollary}{Corollary}[section]
\newtheorem{theorem}{Theorem}[section]
\newdefinition{remark}{Remark}[section]

\newcommand{\rank}{\mathop{\rm rank}\nolimits}
In this paper we shall adopt the following notation. Let ${\mathbb
C}^{m\times n} $ be the set of $m$ by $n$ matrices with complex
entries and ${\rm \bf I}_{m}$ be the identity matrix of order $m$.
Denote by ${\rm {\bf a}}_{.j} $ and ${\rm {\bf a}}_{i.} $ the
$j$th column  and the $i$th row of ${\rm {\bf A}}\in {\mathbb
C}^{m\times n} $, respectively.  Let ${\rm {\bf A}}_{.j} \left( {{\rm {\bf b}}} \right)$
denote the matrix obtained from ${\rm {\bf A}}$ by replacing its
$j$th column with the  vector ${\rm {\bf b}}$, and by ${\rm {\bf
A}}_{i.} \left( {{\rm {\bf b}}} \right)$ denote the matrix
obtained from ${\rm {\bf A}}$ by replacing its $i$th row with
${\rm {\bf b}}$.

Let $\alpha : = \left\{ {\alpha _{1} ,\ldots ,\alpha _{k}}
\right\} \subseteq {\left\{ {1,\ldots ,m} \right\}}$ and $\beta :
= \left\{ {\beta _{1} ,\ldots ,\beta _{k}}  \right\} \subseteq
{\left\{ {1,\ldots ,n} \right\}}$ be subsets of the order $1 \le k
\le \min {\left\{ {m,n} \right\}}$. Then ${\left| {{\rm {\bf
A}}_{\beta} ^{\alpha} } \right|}$ denotes the minor of ${\rm {\bf
A}}$ determined by the rows indexed by $\alpha$ and the columns
indexed by $\beta$. Clearly, ${\left| {{\rm {\bf A}}_{\alpha}
^{\alpha} } \right|}$ be a principal minor determined by the rows
and  columns indexed by $\alpha$. For $1 \leq k\leq n$, denote by
 \[\textsl{L}_{ k, n}: = {\left\{
{\,\alpha :\alpha = \left( {\alpha _{1} ,\ldots ,\alpha _{k}}
\right),\,{\kern 1pt} 1 \le \alpha _{1} \le \ldots \le \alpha _{k}
\le n} \right\}}\]
 the collection of strictly increasing sequences
of $k$ integers chosen from the set $\left\{ {1,\ldots ,n}
\right\}$. For fixed $i \in \alpha $ and $j \in \beta $, let
\[I_{k,\,m} {\left\{ {i} \right\}}: = {\left\{ {\,\alpha :\alpha
\in L_{k,m} ,i \in \alpha} \right\}}{\rm ,} \quad J_{k,\,n}
{\left\{ {j} \right\}}: = {\left\{ {\,\beta :\beta \in L_{k,n} ,j
\in \beta} \right\}}.\]

Matrix equation is one of the important study fields of linear
algebra. Linear matrix equations, such as
\begin{equation}\label{eq1:AX=B}
 {\rm {\bf A}}{\rm {\bf X}} = {\rm {\bf B}},
\end{equation}
\begin{equation}\label{eq1:XB=D}
 {\rm {\bf X}}{\rm {\bf A}} = {\rm {\bf B}},
\end{equation}
 and \begin{equation}\label{eq1:AXB=D}
 {\rm {\bf A}}{\rm {\bf X}}{\rm {\bf B}} = {\rm {\bf
D}},
\end{equation}
play an important role in linear system theory therefore a large
number of papers have presented several methods for investigating
these matrix equations (for example, see
\cite{da}-\cite{vet}).

In some situations,
however, people pay more attention to the Drazin inverse solutions
of  singular linear systems and matrix equations \cite{ch}-\cite{gu2}. Moreover, Xu Zhao-liang and Wang Guo-rong in  \cite{xu} proved that  the Drazin inverse solutions of the matrix equations (\ref{eq1:AX=B}), (\ref{eq1:XB=D}) and (\ref{eq1:AXB=D}) with some restricts are their unique solutions.
 The Cramer rule for the Drazin inverse solution of the restricted system of linear  equations are used in \cite{wer}- \cite{ha}.
The Cramer rules for solutions of the restricted matrix equations (\ref{eq1:AX=B}), (\ref{eq1:XB=D}) and
(\ref{eq1:AXB=D}), in particular for the Drazin inverse solution,  are established in \cite{wan1}-\cite{gu}.

In this paper, we obtain explicit formulas for determinantal representations of the Drazin inverse
solutions of the matrix equations
(\ref{eq1:AX=B}), (\ref{eq1:XB=D}) and (\ref{eq1:AXB=D}) and using their for determinantal representations of solutions of some differential matrix equations. The paper is based on principles used in  \cite{ky1}, where we obtained analogs
of the Cramer rule for the minimum norm least squares solutions of the matrix
equations (1), (2) and (3).
Liu et al. in  \cite{liu1} deduce the new determinantal
representations of ${\bf A}_{T,S}^{(2)}$ and the Cramer rule for the restricted matrix equation (\ref{eq1:AXB=D})  based on these principles as well. Since  the Drazin inverse and
the group inverse A are outer inverses ${\bf A}_{T,S}^{(2)}$ for some specific choice of $T$ and $S$, then the results obtained in \cite{liu1} generalize in some ways some results of the paper. But we get the more detailed representation of the Drazin inverse solutions, and therefore we can used their for determinantal representations  of solutions of some differential matrix equations.

The paper is organized as follows. We start with some
basic concepts and results about  the Drazin inverse in Section 2.  We use the determinantal representation of the Drazin inverse obtained in \cite{ky} and also  another determinantal representation is obtained in this section.
In Section 3, we derive explicit formulas for determinantal representations of the Drazin inverse
solutions   for the
matrix equations (\ref{eq1:AX=B}), (\ref{eq1:XB=D}) and
(\ref{eq1:AXB=D}). These formulas  generalize the well-known Cramer rule. In Section 4, we demonstrate their using  for determinantal representations  of solutions of some differential matrix equations, ${\bf X}'+ {\bf A}{\bf X}={\bf B}$ and ${\bf X}'+{\bf X}{\bf A}={\bf B}$, where the matrix ${\bf A}$ is singular.
In Section 5, we show  numerical examples to illustrate the main results.

\section{Determinantal representations of the Drazin
inverse} For any  matrix  ${\rm {\bf A}}\in {\mathbb C}^{n\times
n} $ with $ Ind{\kern 1pt} {\rm {\bf A}}=k$, where  a positive
integer $k =: Ind{\kern 1pt} {\rm {\bf A}}= {\mathop {\min}
\limits_{k \in N \cup {\left\{ {0} \right\}}} }{\kern 1pt}
{\left\{ {\rank{\rm {\bf A}}^{k + 1} = \rank{\rm {\bf A}}^{k}}
\right\}}$, the Drazin inverse  is the unique matrix ${\rm {\bf
X}}$ that satisfies the following three properties
 \begin{equation}\label{eq:Dr_prop}
\begin{array}{l}
  1)\,\,  {\rm {\bf A}}^{k+1}{\rm
{\bf X}}={\rm {\bf
         A}}^{k};\\
  2)\,\,{\rm {\bf X}}{\rm {\bf A}}{\rm {\bf X}}={\rm {\bf X}};\\
  3)\,\, {\rm
{\bf A}}{\rm {\bf X}}={\rm {\bf X}}{\rm {\bf A}}.
\end{array}
\end{equation}
It is denoted by ${\rm {\bf X}}={\rm {\bf A}}^{D}$.

 In particular, when $Ind{\kern 1pt} {\rm {\bf A}}=1$,
then the matrix ${\rm {\bf X}}$ in (\ref{eq:Dr_prop}) is called
the group inverse and is denoted by ${\rm {\bf X}}={\rm {\bf
A}}^{g }$.

If $Ind{\kern 1pt} {\rm {\bf A}}=0$, then ${\rm {\bf A}}$ is
nonsingular, and ${\rm {\bf A}}^{D}\equiv {\rm {\bf A}}^{-1}$.

\begin{remark} Since the equation 3) of (\ref{eq:Dr_prop}), the equation 1)  can be replaced by follows
\[ 1a)\,\,  {\rm {\bf X}}{\rm {\bf A}}^{k+1}={\rm {\bf
         A}}^{k}.\]
\end{remark}
The Drazin inverse can be represented explicitly by the Jordan
canonical form as follows.
\begin{theorem}\cite{ca} If ${\rm {\bf A}} \in {\mathbb C}^{n\times n}$ with
$Ind{\kern 1pt} {\rm {\bf A}} = k $ and
\[
{\rm {\bf A}} = {\rm {\bf P}}\begin{pmatrix}
  {\rm {\bf C}} & {\rm {\bf 0}} \\
  {\rm {\bf 0}}& {\rm {\bf N}}
\end{pmatrix} {\rm {\bf P}}^{ - 1}
\]
where ${\rm {\bf C}}$ is nonsingular and $\rank{\rm {\bf C}} =
\rank{\rm {\bf A}}^{k}$, and ${\rm {\bf N}}$ is nilpotent of order
$k$, then
\[
{\rm {\bf A}}^{D} = {\rm {\bf P}}\begin{pmatrix}
  \rm {\bf C}^{ - 1} & {\rm {\bf 0}} \\
  {\rm {\bf 0}}& {\rm {\bf 0}}
\end{pmatrix} {\rm {\bf P}}^{ - 1}.
\]
\end{theorem}

We use the following theorem about the limit representation of the
Drazin inverse.
\begin{theorem} \cite{ca}\label{theor:lim_draz} If ${\rm {\bf A}}
 \in {\mathbb C}^{n\times n}$, then
\[
{\rm {\bf A}}^{D} = {\mathop {\lim} \limits_{\lambda \to 0}}
\left( {\lambda {\rm {\bf I}}_n + {\rm {\bf A}}^{k + 1}} \right)^{
- 1}{\rm {\bf A}}^{k},
\]
where $k = Ind{\kern 1pt} {\rm {\bf A}}$, $\lambda \in {\mathbb R}
_{ +}  $, and ${\mathbb R} _{ +} $ is a set of the real positive
numbers.
\end{theorem}
The following  theorem can be obtained by analogy to Theorem
\ref{theor:lim_draz}.
\begin{theorem} \label{theor:lim_A_k_A_k+1} If ${\rm {\bf A}}
 \in {\mathbb C}^{n\times n}$, then
\[
{\rm {\bf A}}^{D} = {\mathop {\lim} \limits_{\lambda \to 0}}{\rm
{\bf A}}^{k} \left( {\lambda {\rm {\bf I}}_n + {\rm {\bf A}}^{k +
1}} \right)^{ - 1},
\]
where $k = Ind{\kern 1pt} {\rm {\bf A}}$, $\lambda \in {\mathbb R}
_{ +}  $, and ${\mathbb R} _{ +} $ is a set of the real positive
numbers.
\end{theorem}
 Denote by ${\rm {\bf a}}_{.j}^{(k)} $ and ${\rm {\bf
a}}_{i.}^{(k)} $ the $j$th column  and the $i$th row of ${\rm {\bf
A}}^{k} $ respectively.
\begin{lemma}(\cite{ky}, Lemma 3.1) \label{lem:rank_A_col} If ${\rm {\bf A}} \in {\mathbb C}^{n\times n}$
with $Ind{\kern 1pt} {\rm {\bf A}} = k $, then for all $i,j =
\overline {1,n}$
\[
\label{eq:rank_column} \rank{\rm {\bf A}}_{.{\kern 1pt} i}^{k + 1}
\left( {{\rm {\bf a}}_{.j}^{\left( {k} \right)}}  \right) \le
\rank{\rm {\bf A}}^{k + 1}.
\]
\end{lemma}
 Using Theorem \ref{theor:lim_draz} and Lemma
\ref{lem:rank_A_col}  we obtained in \cite{ky} the following determinantal
 representations of the Drazin and group inverses and the identity ${\bf A}^{D}{\bf A}$   on  $R  ({\bf A}^{k})$.
\begin{theorem}(\cite{ky}, Theorem 3.3)\label{theor:dr_repr_col}
 If $Ind{\kern 1pt} {\rm {\bf A}} = k $ and $\rank{\rm {\bf
A}}^{k + 1} = \rank{\rm {\bf A}}^{k}=r \le n$ for  ${\rm {\bf A}}\in {\mathbb C}^{n\times n} $, then the
Drazin inverse  ${\rm {\bf A}}^{D} = \left( {a_{ij}^{ D} } \right)
\in {\rm {\mathbb{C}}}_{}^{n\times n} $ possess the following
determinantal representations:
\begin{equation}
\label{eq:dr_repr_col}
  a_{ij}^{ D} = \frac{
  {\sum\limits_{\beta \in J_{r,n} {\left\{ {i} \right\}}}
{{\left| {\left( {{\rm {\bf A}}_{.\,i}^{k + 1} \left( {{\rm {\bf
a}}_{.j}^{\left( {k} \right)}}  \right)} \right)_{\beta} ^{\beta}
} \right|}}}  }{
  {\sum\limits_{\beta \in J_{r,n}} {{\left| {\left( {{\rm {\bf
A}}^{k + 1}} \right)_{\beta} ^{\beta }} \right|}}}},
 \end{equation}
for all $i,j = \overline {1,n}$.
\end{theorem}
\begin{corollary}(\cite{ky}, Corollary 3.1)
 If $Ind{\kern 1pt} {\rm {\bf A}} = 1 $ and $\rank{\rm {\bf
A}}^{2} = \rank{\rm {\bf A}}=r \le n$ for 
${\rm {\bf A}}\in {\mathbb C}^{n\times n} $, then the
group inverse  ${\rm {\bf A}}^{g} = \left( {a_{ij}^{ g} } \right)
\in {\rm {\mathbb{C}}}_{}^{n\times n} $ possess the following
determinantal representation:
\[
 a_{ij}^{g} ={\frac{\sum\limits_{\beta \in J_{r,n} {\left\{ {i} \right\}}} {{\left|
{\left( {{\rm {\bf A}}_{.\,i}^{2} \left( {{\rm {\bf a}}_{.j}}
\right)} \right)_{\beta} ^{\beta} } \right|}}}{{\sum\limits_{\beta \in J_{r,n}} {{\left| {\left( {{\rm {\bf
A}}^{2}} \right)_{\beta} ^{\beta }}  \right|}}}}},
\]
 for all $i,j = \overline {1,n}. $
\end{corollary}
\begin{corollary}(\cite{ky}, Corollary 3.2)
If $Ind{\kern 1pt} {\rm {\bf A}} = k $ and $\rank{\rm {\bf A}}^{k
+ 1} = \rank{\rm {\bf A}}^{k}=r \le n$ for
${\rm {\bf A}}\in {\mathbb C}^{n\times n} $, then the matrix ${\bf A}^{D}{\bf A}= (p_{ij})\in {\mathbb C}^{n\times n}$ possess the following determinantal representation
\begin{equation}\label{eq:AdA}
 p_{ij} =\frac{{\sum\limits_{\beta \in J_{r,n} {\left\{ {i} \right\}}} {{\left|
{\left( {{\rm {\bf A}}_{.\,i}^{k+1} \left( {{\rm {\bf
a}}_{.j}}^{(k+1)} \right)} \right)_{\beta} ^{\beta} } \right|}}}}{{\sum\limits_{\beta \in J_{r,n} } {{\left|
{\left( {{\rm {\bf A}}_{.\,i}^{k+1} } \right)_{\beta} ^{\beta} } \right|}}}},
\end{equation}
for all $i,j = \overline {1,n}. $
\end{corollary}

Using Theorem \ref{theor:lim_A_k_A_k+1} we can obtain another
determinantal representation of the Drazin inverse. At first we
consider the following auxiliary lemma similar to Lemma  \ref{lem:rank_A_col}.
\begin{lemma} \label{lem:rank_A_row} If ${\rm {\bf A}} \in {\mathbb C}^{n\times n}$
with $Ind{\kern 1pt} {\rm {\bf A}} = k $, then for all $i,j =
\overline {1,n}$
\[ \rank{\rm {\bf A}}_{i{\kern 1pt} .}^{k + 1}
\left( {{\rm {\bf a}}_{j.}^{\left( {k} \right)}}  \right) \le
\rank{\rm {\bf A}}^{k + 1}.
\]
\end{lemma}

\newproof{proof}{Proof}
\begin{proof}  The matrix  $ {\rm {\bf A}}^{ k+1} _{i\,.} \left({{\rm {\bf a}}_{j\,.}^{ (k)} }  \right)$ may by represent as follows
\[\left(
    \begin{array}{ccc}
      {{\sum\limits_{s=1}^{n}a_{1s} {a_{s1}^{ (k)}  } } } & \ldots & {{\sum\limits_{s=1}^{n}a_{1s} {a_{sn}^{ (k)}  } } } \\
      \ldots & \ldots & \ldots \\
      a_{j1}^{ (k)}  & \ldots & a_{j\,n}^{ (k)}  \\
      \ldots & \ldots & \ldots \\
      {{\sum\limits_{s=1}^{n}a_{ns} {a_{s1}^{ (k)}  } } } & \ldots & {{\sum\limits_{s=1}^{n}a_{n\,s} {a_{sn}^{ (k)}  } } } \\
    \end{array}
  \right)
\]
Let ${\rm {\bf P}}_{l\,i} \left( {-a_{l\,j}}  \right)\in {\mathbb
C}^{n\times n} $, $(l \ne i )$, be a matrix with $-a_{l\,j} $ in
the $(l, i)$ entry, 1 in all diagonal entries, and 0 in others.
It is a matrix of an
elementary transformation. It follows that
\[
{\rm {\bf A}}^{ k+1} _{i\,.} \left({{\rm {\bf a}}_{j\,.}^{ (k)} }  \right) \cdot \prod\limits_{l \ne
i}{\rm {\bf P}}_{l\,i} \left( {-a_{l\,j}}  \right) =\left({
    \begin{array}{ccc}
      {{\sum\limits_{s\neq j}^{n} a_{1s} {a_{s1}^{ (k)}  } } } & \ldots & {{\sum\limits_{s\neq j}^{n}a_{1s} {a_{sn}^{ (k)}  } } } \\
      \ldots & \ldots & \ldots \\
      a_{j1}^{ (k)}  & \ldots & a_{j\,n}^{ (k)}  \\
      \ldots & \ldots & \ldots \\
      {{\sum\limits_{s\neq j}^{n} a_{ns} {a_{s1}^{ (k)}  } } } & \ldots & {{\sum\limits_{s\neq j}^{n}a_{ns} {a_{sn}^{ (k)}  } } } \\
    \end{array}}
  \right)\, ith
\]
 The obtained above matrix  has the following factorization.
\[
 \begin{pmatrix}
      {{\sum\limits_{s\neq j}^{n} a_{1s} {a_{s1}^{ (k)}  } } } & \ldots & {{\sum\limits_{s\neq j}^{n}a_{1s} {a_{sn}^{ (k)}  } } } \\
      \ldots & \ldots & \ldots \\
      a_{j1}^{ (k)}  & \ldots & a_{j\,n}^{ (k)}  \\
      \ldots & \ldots & \ldots \\
      {{\sum\limits_{s\neq j}^{n} a_{ns} {a_{s1}^{ (k)}  } } } & \ldots & {{\sum\limits_{s\neq j}^{n}a_{ns} {a_{sn}^{ (k)}  } } } \\
 \end{pmatrix} =
\]
\[ \begin{pmatrix}
      a_{11} &  \ldots & 0 &  \ldots & a_{1n} \\
       \ldots &  \ldots &  \ldots &  \ldots &  \ldots \\
      0 &  \ldots & 1 &  \ldots & 0 \\
       \ldots &  \ldots &  \ldots &  \ldots &  \ldots \\
      a_{n1} &  \ldots & 0 &  \ldots & a_{nn} \\
  \end{pmatrix}
    \begin{pmatrix}
      a_{11}^{(k)} &  a_{12}^{(k)}  & \ldots &  a_{1n}^{(k)}  \\
      a_{21}^{(k)}  &  a_{22}^{(k)}  & \ldots &  a_{2n}^{(k)}  \\
      \ldots & \ldots & \ldots & \ldots \\
       a_{n1}^{(k)}  &  a_{n2}^{(k)}  & \ldots &  a_{nn}^{(k)}  \\
    \end{pmatrix}
\]
 Denote the first matrix by
  \[{\rm {\bf \tilde {A}}}: = {\mathop
{\left( {{\begin{array}{*{20}c}
 {a_{11}}  \hfill & {\ldots}  \hfill & {0} \hfill & {\ldots}  \hfill &
{a_{1n}}  \hfill \\
 {\ldots}  \hfill & {\ldots}  \hfill & {\ldots}  \hfill & {\ldots}  \hfill &
{\ldots}  \hfill \\
 {0} \hfill & {\ldots}  \hfill & {1} \hfill & {\ldots}  \hfill & {0} \hfill
\\
 {\ldots}  \hfill & {\ldots}  \hfill & {\ldots}  \hfill & {\ldots}  \hfill &
{\ldots}  \hfill \\
 {a_{n1}}  \hfill & {\ldots}  \hfill & {0} \hfill & {\ldots}  \hfill &
{a_{nn}}  \hfill \\
\end{array}} } \right)}\limits_{jth}} ith.\]
The matrix
${\rm {\bf \tilde {A}}}$ is obtained from ${\rm {\bf A}}$ by
replacing all entries of the $i$th row  and the $j$th column with
zeroes except for 1 in the $(i, j)$ entry. Elementary
transformations of a matrix do not change its rank. It follows that
$\rank{\rm {\bf A}}^{ k+1} _{i\,.} \left({{\rm {\bf a}}_{j\,.}^{ (k)} }  \right) \le \min {\left\{
{\rank{\rm {\bf A}}^{ k },\rank{\rm {\bf \tilde {A}}}} \right\}}$.
Since $\rank{\rm {\bf \tilde {A}}} \ge \rank{\rm {\bf A}}^{ k} $ the proof is completed.
\end{proof}
\begin{theorem}\label{theor:dr_repr_row}
 If $Ind{\kern 1pt} {\rm {\bf A}} = k $ and $\rank{\rm {\bf
A}}^{k + 1} = \rank{\rm {\bf A}}^{k}=r \le n$ for  ${\rm {\bf A}}\in {\mathbb C}^{n\times n} $, then the
Drazin inverse  ${\rm {\bf A}}^{D} = \left( {a_{ij}^{ D} } \right)
\in {\rm {\mathbb{C}}}_{}^{n\times n} $ possess the following
determinantal representations:
\begin{equation}
\label{eq:dr_repr_row}
  a_{ij}^{ D} = \frac{
  {\sum\limits_{\alpha \in I_{r,n} {\left\{ {j} \right\}}}
{{\left| {\left( {{\rm {\bf A}}_{j.}^{k + 1} \left( {{\rm {\bf
a}}_{i.}^{\left( {k} \right)}}  \right)} \right)_{\alpha}
^{\alpha} } \right|}}}  }
{  {\sum\limits_{\alpha \in I_{r,n}} {{\left| {\left( {{\rm {\bf
A}}^{k + 1}} \right)_{\alpha} ^{\alpha }} \right|}}}},
 \end{equation}
for all $i,j = \overline {1,n}$.
\end{theorem}
\begin{proof}
If $\lambda\in {\mathbb{R}}_{+}$, then  $\rank \left( {\lambda {\rm {\bf I}} + {\rm {\bf A}}^{ k+1} } \right)=n$.
 Hence,   there exists the  inverse matrix
\[
\left( {\lambda {\rm {\bf I}} + {\rm {\bf A}}^{ k+1} }
\right)^{ - 1} = {\frac{{1}}{{\det \left( {\lambda {\rm {\bf I}} +
{\rm {\bf A}}^{ k+1 }} \right)}}}\left(
{{\begin{array}{*{20}c}
 {R_{11}}  \hfill & {R_{21}}  \hfill & {\ldots}  \hfill & {R_{n\,1}}  \hfill
\\
 {R_{12}}  \hfill & {R_{22}}  \hfill & {\ldots}  \hfill & {R_{n\,2}}  \hfill
\\
 {\ldots}  \hfill & {\ldots}  \hfill & {\ldots}  \hfill & {\ldots}  \hfill
\\
 {R_{1\,n}}  \hfill & {R_{2\,n}}  \hfill & {\ldots}  \hfill & {R_{n\,n}}  \hfill
\\
\end{array}} } \right),
\]
 where  $R_{ij} $  is a cofactor
  in $\lambda {\rm {\bf
I}} + {\rm {\bf A}}^{ k+1}$ for all $ i,j=\overline{1,n}$.   By Theorem \ref{theor:lim_A_k_A_k+1},
${\rm {\bf A}}^{D} = {\mathop {\lim} \limits_{\lambda \to 0}}{\rm
{\bf A}}^{k} \left( {\lambda {\rm {\bf I}}_n + {\rm {\bf A}}^{k +
1}} \right)^{ - 1}$, so that
\[{\rm {\bf A}}^{ D}  ={\mathop {\lim} \limits_{\lambda \to 0}}{\frac{{1}}{{\det \left( {\lambda {\rm {\bf I}} +
{\rm {\bf A}}^{ k+1 }} \right)}}}
                                       \begin{pmatrix}
                                         \sum_{s=1}^{n} a_{1s}^{(k)} R_{1s} & \ldots & \sum_{s=1}^{n} a_{1s}^{(k)} R_{ns}  \\
                                         \ldots &  \ldots &  \ldots \\
                                         \sum_{s=1}^{n} a_{ns}^{(k)} R_{1s} &  \ldots & \sum_{s=1}^{n} a_{ns}^{(k)} R_{ns} \\  \end{pmatrix}=\]
\begin{equation}
\label{lim_A_D}
{\mathop {\lim} \limits_{\lambda \to 0}}\begin{pmatrix}
 \frac{\det \left( {\lambda {\rm {\bf I}} + {\rm {\bf A}}^{ k+1}}
\right)_{1.} \left( {{\rm {\bf a}}_{1.}^{ (k)} }  \right)}{{\det
\left( {\lambda {\rm {\bf I}} + {\rm {\bf A}}^{ k+1} }
\right)}}
 & \ldots & \frac{\det \left( {\lambda {\rm {\bf I}} + {\rm {\bf
A}}^{k+1} } \right)_{n.} \left( {{\rm {\bf
a}}_{n.}^{(k)} }  \right)}{{\det \left( {\lambda {\rm {\bf I}} +
{\rm {\bf
A}}^{ k+1} } \right)}}\\
  \ldots & \ldots & \ldots \\
\frac{\det \left( {\lambda {\rm {\bf I}} + {\rm {\bf A}}^{ k+1}} \right)_{1.} \left( {{\rm {\bf a}}_{n.}^{ (k)} }
\right)}{{\det \left( {\lambda {\rm {\bf I}} + {\rm {\bf A}}^{ k+1}
} \right)}} &
 \ldots &
 \frac{\det \left( {\lambda {\rm {\bf I}} + {\rm {\bf A}}^{ k+1} } \right)_{n.} \left( {{\rm {\bf a}}_{n.}^{ (k)} }
\right)}{{\det \left( {\lambda {\rm {\bf I}} + {\rm {\bf A}}^{k+1}
} \right)}}
\end{pmatrix}
\end{equation}
Similar to the  characteristic polynomial, we have
\[ \det \left( {\lambda{\rm {\bf I}} + {\rm {\bf A}}^{k+1}
} \right) = \lambda ^{n} + d_{1} \lambda ^{n - 1} +
d_{2} \lambda ^{n - 2} + \ldots + d_{n},
\]
 where $d_{s}={  {\sum\limits_{\alpha \in I_{s,n}} {{\left| {\left( {{\rm {\bf
A}}^{k + 1}} \right)_{\alpha} ^{\alpha }} \right|}}}}$  is a
sum of the principal minors  of ${\rm {\bf A}}^{ k+1} $
of order $s$,  for all $ s=\overline{1,n-1}$, and $d_{n}=\det  {\rm {\bf A}}^{k+1} $.
Since $\rank{\rm {\bf A}}^{ k+1} = r$, then $d_{n} = d_{n - 1} = \ldots = d_{r + 1} = 0$ and
\begin{equation}
\label{eq_detA}\det \left( {\lambda {\rm {\bf I}} + {\rm {\bf A}}^{k+1} } \right)
= \lambda ^{n} + d_{1} \lambda ^{n - 1} + d_{2} \lambda ^{n - 2} + \ldots + d_{r} \lambda
^{n - r}.
 \end{equation}
Similarly we have for  all $i,j=\overline{1,n} $
\[ \det \left( {\lambda {\rm {\bf I}} + {\rm {\bf
A}}^{k+1}} \right)_{j.} \left( {{\rm {\bf
a}}_{i.}^{ (k)} } \right) = l_{1}^{\left( {ij} \right)} \lambda ^{n
- 1} + l_{2}^{\left( {ij} \right)} \lambda ^{n - 2} + \ldots +
l_{n}^{\left( {ij} \right)}, \]
 where for  all $s=\overline{1,n-1} $,
 \[l_{s}^{\left( {ij}
\right)} = {\sum\limits_{\alpha \in I_{s,n} {\left\{ {j} \right\}}}
{{\left| {\left( {{\rm {\bf A}}_{j\,.}^{k + 1} \left( {{\rm {\bf
a}}_{i.}^{\left( {k} \right)}}  \right)} \right)_{\alpha}
^{\alpha} } \right|}}},\] and \,\,$l_{n}^{\left( {i\,j} \right)} = \det
{\rm {\bf A}}^{ k+1}_{j\,.} \left(
{{\rm {\bf a}}_{i.}^{\left( {k} \right)} } \right)$.

By Lemma \ref{lem:rank_A_row},
$ \rank{\rm {\bf A}}_{j \,.}^{k + 1}
\left( {{\rm {\bf a}}_{i\,.}^{\left( {k} \right)}}  \right) \le r$, so that if
$s>r$, then for all $\alpha \in I_{s,n} {\left\{
{i} \right\}}$ and for all $i,j = \overline {1,n}$,
\[{{\left| {\left( {{\rm {\bf A}}_{j\,.}^{k + 1} \left( {{\rm {\bf
a}}_{i.}^{\left( {k} \right)}}  \right)} \right)_{\alpha}
^{\alpha} } \right|}}= 0.\]

 Therefore if $r + 1 \le s < n$, then for all $
i,j = \overline {1,n}$,
\[l_{s}^{\left( {ij}
\right)} =  {\sum\limits_{\alpha \in I_{s,n} {\left\{ {j} \right\}}}
{{\left| {\left( {{\rm {\bf A}}_{j.}^{k + 1} \left( {{\rm {\bf
a}}_{i.}^{\left( {k} \right)}}  \right)} \right)_{\alpha}
^{\alpha} } \right|}}}= 0,\]
and $l_{n}^{\left( {ij} \right)} =
\det {\rm {\bf A}}^{ k+1}_{j\,.}
\left( {{\rm {\bf a}}_{i.}^{ (k)} } \right) = 0$.
Finally we obtain
\begin{equation}\label{eq_det_A_j}
  \det \left( {\lambda {\rm {\bf I}} + {\rm {\bf A}}^{k+1} } \right)_{j.} \left( {{\rm {\bf a}}_{i.}^{(k)} } \right) =
l_{1}^{\left( {i\,j} \right)} \lambda ^{n - 1} + l_{2}^{\left(
{i\,j} \right)} \lambda ^{n - 2} + \ldots + l_{r}^{\left( {ij}
\right)} \lambda ^{n - r}.
\end{equation}

 By replacing the denominators and the nominators  of the fractions in the entries of the matrix (\ref{lim_A_D}) with  the
 expressions (\ref{eq_detA}) and (\ref{eq_det_A_j}) respectively, finally we obtain
\[
   {\rm {\bf A}}^{D}  = {\mathop {\lim} \limits_{\lambda \to 0}}
\begin{pmatrix}
 {{\frac{{l_{1}^{\left( {11} \right)} \lambda ^{n - 1} + \ldots +
l_{r}^{\left( {11} \right)} \lambda^{n - r}}}{{\lambda ^{n} +
d_{1} \lambda ^{n - 1} + \ldots + d_{r} \lambda ^{n - r}}}}}
 & \ldots & {{\frac{{l_{1}^{\left( {1n} \right)} \lambda ^{n
- 1} + \ldots + l_{r}^{\left( {1n} \right)} \lambda ^{n -
r}}}{{\lambda ^{n} + d_{1} \lambda
^{n - 1} + \ldots + d_{r} \lambda ^{n - r}}}}} \\
  \ldots & \ldots & \ldots \\
  {{\frac{{l_{1}^{\left( {n1} \right)} \lambda ^{n - 1} + \ldots +
l_{r}^{\left( {n1} \right)} \lambda ^{n - r}}}{{\lambda ^{n} +
d_{1} \lambda ^{n - 1} + \ldots + d_{r} \lambda ^{n - r}}}}} &
\ldots & {{\frac{{l_{1}^{\left( {nn} \right)} \lambda ^{n - 1} +
\ldots + l_{r}^{\left( {nn} \right)} \lambda ^{n - r}}}{{\lambda
^{n} + d_{1} \lambda ^{n - 1} + \ldots + d_{r} \lambda ^{n -
r}}}}}
\end{pmatrix}=\]
  \[ = \left( {{\begin{array}{*{20}c}
 {{\frac{{l_{r}^{\left( {11} \right)}} }{{d_{r}} }}} \hfill & {\ldots}
\hfill & {{\frac{{l_{r}^{\left( {1n} \right)}} }{{d_{r}} }}} \hfill \\
 {\ldots}  \hfill & {\ldots}  \hfill & {\ldots}  \hfill \\
 {{\frac{{l_{r}^{\left( {n1} \right)}} }{{d_{r}} }}} \hfill & {\ldots}
\hfill & {{\frac{{l_{r}^{\left( {nn} \right)}} }{{d_{r}} }}} \hfill \\
\end{array}} } \right),
\]
where for all $i,j=\overline{1,n} $,
 \[l_{r}^{\left( {ij}
\right)} = {\sum\limits_{\alpha \in I_{r,n} {\left\{ {j} \right\}}}
{{\left| {\left( {{\rm {\bf A}}_{j.}^{k + 1} \left( {{\rm {\bf
a}}_{i.}^{\left( {k} \right)}}  \right)} \right)_{\alpha}
^{\alpha} } \right|}}}.\]
This completes the proof.
\end{proof}

Using Theorem \ref{theor:dr_repr_row} we evidently can obtain another determinantal representation of the group inverse and  the following determinantal representation of the  identity ${\bf A}{\bf A}^{D}$   on  $R  ({\bf A}^{k})$
 \begin{corollary}
 If $Ind{\kern 1pt} {\rm {\bf A}} = 1 $ and $\rank{\rm {\bf
A}}^{2} = \rank{\rm {\bf A}}=r \le n$ for
${\rm {\bf A}}\in {\mathbb C}^{n\times n} $, then the
group inverse  ${\rm {\bf A}}^{g} = \left( {a_{ij}^{ g} } \right)
\in {\rm {\mathbb{C}}}_{}^{n\times n} $ possess the following
determinantal representations:
\begin{equation}\label{eq:AAg}
 a_{ij}^{g} ={\frac{\sum\limits_{\alpha \in I_{r,n} {\left\{ {j} \right\}}} {{\left|
{\left( {{\rm {\bf A}}_{j.}^{2} \left( {{\rm {\bf a}}_{i.}}
\right)} \right)_{\alpha} ^{\alpha} } \right|}}}{{\sum\limits_{\alpha \in I_{r,n}} {{\left| {\left( {{\rm {\bf
A}}^{2}} \right)_{\alpha} ^{\alpha }}  \right|}}}}},
 \end{equation}
 for all $i,j = \overline {1,n}. $
\end{corollary}

\begin{corollary}
If $Ind{\kern 1pt} {\rm {\bf A}} = k $ and $\rank{\rm {\bf A}}^{k
+ 1} = \rank{\rm {\bf A}}^{k}=r \le n$ for 
${\rm {\bf A}}\in {\mathbb C}^{n\times n} $, then the matrix ${\bf A}{\bf A}^{D}= (q_{ij})\in {\mathbb C}^{n\times n}$ possess the following determinantal representation
\begin{equation}\label{eq:AAd}
  q_{ij}= \frac{
  {\sum\limits_{\alpha \in I_{r,n} {\left\{ {j} \right\}}}
{{\left| {\left( {{\rm {\bf A}}_{j.}^{k + 1} \left( {{\rm {\bf
a}}_{i.}^{\left( {k+1} \right)}}  \right)} \right)_{\beta}
^{\beta} } \right|}}}  }
{  {\sum\limits_{\alpha \in I_{r,n}} {{\left| {\left( {{\rm {\bf
A}}^{k + 1}} \right)_{\beta} ^{\beta }} \right|}}}},
 \end{equation}
for all $i,j = \overline {1,n}. $
\end{corollary}

\section{Cramer's rule of the Drazin inverse solutions of some  matrix
equations}
Consider a  matrix equation
\begin{equation}\label{eq:AX=B}
 {\rm {\bf A}}{\rm {\bf X}} = {\rm {\bf B}},
\end{equation}
where ${\rm {\bf
A}}\in  {\mathbb{C}}^{n\times n} $ with $Ind{\kern 1pt} {\rm {\bf A}}= k $, ${\rm {\bf B}}\in  {\mathbb{C}}^{n\times m}$ are given and ${\rm {\bf X}} \in
{\mathbb{C}}^{n\times m}$ is unknown.
\begin{theorem}(\cite{xu}, Theorem 1) If  the range space $R({\bf B})\subset R({\bf A}^{k})$,
then the matrix equation (\ref{eq:AX=B}) with constrain
$R({\bf X})\subset R({\bf A}^{k})$ has a unique solution
\[ {\bf X}={\bf A}^{D}{\bf B}.\]
\end{theorem}
 We denote ${\rm {\bf A}}^{ k}{\rm {\bf B}}=:\hat{{\rm
{\bf B}}}= (\hat{b}_{ij})\in {\mathbb{C}}^{n\times m}$.
\begin{theorem}
If $\rank{\rm {\bf
A}}^{k + 1} = \rank{\rm {\bf A}}^{k}=r \le n$ for ${\rm {\bf A}}\in {\mathbb C}^{n\times n} $, then for Drazin inverse solution $ {\bf X}={\bf A}^{D}{\bf B}= (x_{ij})\in {\mathbb{C}}^{n\times m}$ of  (\ref{eq:AX=B}) we have for all $i=\overline {1,n} $, $j=\overline {1,m} $,
\begin{equation}
\label{eq:dr_AX} x_{ij} = {\frac{{{\sum\limits_{\beta \in
J_{r,\,n} {\left\{ {i} \right\}}} {{\left| \left( {{\rm
{\bf A}}^{ k+1}_{\,.\,i} \left( {\hat{{\rm
{\bf b}}}_{.j}} \right)} \right) {\kern 1pt} _{\beta}
^{\beta} \right|}} } }}{{{\sum\limits_{\beta \in J_{r,\,\,n}}
{{\left| {\left( {{\rm {\bf A}}^{ k+1} } \right){\kern
1pt} {\kern 1pt} _{\beta} ^{\beta} }  \right|}}} }}}.
\end{equation}
\end{theorem}

\begin{proof}  By Theorem  \ref{theor:dr_repr_col} we can represent the matrix
${\rm {\bf A}}^{ D} $ by (\ref{eq:dr_repr_col}). Therefore, we
obtain for all $i=\overline {1,n} $, $j=\overline {1,m} $,
\[
x_{ij}  =\sum_{s=1}^{n}
a_{is}^{D}b_{sj}=\sum_{s=1}^{n}{\frac{{{\sum\limits_{\beta \in
J_{r,\,n} {\left\{ {i} \right\}}} {{\left| \left( {{\rm
{\bf A}}^{ k+1}_{\,. \,i} \left( {{\rm {\bf
a}}_{\,.s}^{ (k)} }  \right)} \right) {\kern 1pt} _{\beta}
^{\beta}\right|} } } }}{{{\sum\limits_{\beta \in J_{r,\,n}}
{{\left| {\left( {{\rm {\bf A}}^{ k+1} } \right){\kern
1pt} _{\beta} ^{\beta} }  \right|}}} }}}\cdot b_{sj}=
\]
\[
{\frac{{{\sum\limits_{\beta \in J_{r,\,n} {\left\{ {i}
\right\}}}\sum_{s=1}^{n} {{\left| \left( {{\rm {\bf A}}^{
k+1} _{\,. \,i} \left( {{\rm {\bf a}}_{\,.s}^{
(k)} } \right)} \right) {\kern 1pt} _{\beta} ^{\beta}\right|} } }
}\cdot b_{sj}} {{{\sum\limits_{\beta \in J_{r,\,n}} {{\left|
{\left( {{\rm {\bf A}}^{ k+1} } \right){\kern 1pt}
_{\beta} ^{\beta} } \right|}}} }}}.
\]
Since ${\sum\limits_{s} {{\rm {\bf a}}_{\,.s}^{ (k)}  b_{sj}} }=
\left( {{\begin{array}{*{20}c}
 {{\sum\limits_{s} {a_{1s}^{(k)}  b_{sj}} } } \hfill \\
 {{\sum\limits_{s} {a_{2s}^{ (k)}  b_{sj}} } } \hfill \\
 { \vdots}  \hfill \\
 {{\sum\limits_{s} {a_{ns}^{(k)}  b_{sj}} } } \hfill \\
\end{array}} } \right) = \hat{{\rm {\bf b}}}_{.j}$, then it follows (\ref{eq:dr_AX}).
\end{proof}
\begin{corollary}(\cite{ky}, Theorem 4.2.)  If $Ind{\kern 1pt} {\rm {\bf A}} = k $ and $\rank{\rm {\bf
A}}^{k + 1} = \rank{\rm {\bf A}}^{k}=r \le n$ for ${\rm {\bf A}}\in {\mathbb C}^{n\times n} $, and ${\rm {\bf y}} = \left( {y_{1},\ldots ,y_{n} } \right)^{T}\in {\mathbb C}^{n}$, then for Drazin inverse solution $ {\bf x}={\bf A}^{D}{\bf y}=:  \left( {x_{1},\ldots ,x_{n} } \right)^{T}\in {\mathbb C}^{n}$ of the system of  linear equations
\[{\rm {\bf A}} \cdot {\rm {\bf x}} = {\rm {\bf y}},\]
we have for all $j = \overline {1,n}$,

\[
x_{j} = {\frac{{\sum\limits_{\beta \in J_{r,n}
{\left\{ {j} \right\}}} {{\left| {\left( {{\rm {\bf A}}^{k+1}_{.\,j} ({\rm {\bf f}}  )} \right)_{\beta} ^{\beta} }
\right|}}}}{{\left|  \left( {{\rm {\bf A}}^{k+1} }
\right)_{\beta} ^{\beta}
\right| } }},
\]
where  ${\rm {\bf f}} = {\rm {\bf A}}^{ k} {\rm {\bf y}}$.
\end{corollary}

Consider a  matrix equation
\begin{equation}\label{eq:XA=B}
 {\rm {\bf X}}{\rm {\bf A}} = {\rm {\bf B}},
\end{equation}
 where ${\rm {\bf
A}}\in{\mathbb{C}}^{m\times m}$ with  $Ind{\kern 1pt} {\rm {\bf A}} = k $, ${\rm {\bf B}}\in {\mathbb{C}}^{n\times m} $ are given and  ${\rm {\bf X}} \in
{\mathbb{C}}^{n\times m}$ is unknown.
\begin{theorem}(\cite{xu}, Theorem 2) If the null space $N({\bf B})\supset N({\bf A}^{k})$, then
the matrix equation (\ref{eq:XA=B}) with constrain
$N({\bf X})\supset N({\bf A}^{k})$ has a unique solution
\[ {\bf X}={\bf B}{\bf A}^{D}.\]
\end{theorem}

 We denote  ${\rm {\bf B}}{\rm {\bf A}}^{
k}=:\check{{\rm {\bf B}}}= (\check{b}_{ij})\in
{\mathbb{C}}^{n\times m}$.
\begin{theorem}
 If $\rank{\rm {\bf
A}}^{k + 1} = \rank{\rm {\bf A}}^{k}=r \le m$ for ${\rm {\bf A}}\in {\mathbb C}^{m\times m} $, then for Drazin inverse solution $ {\bf X}={\bf B}{\bf A}^{D}= (x_{ij})\in {\mathbb C}^{n\times m}$ of  (\ref{eq:XA=B}),  we have for all $i=\overline {1,n} $, $j=\overline {1,m} $,
\begin{equation}
\label{eq:dr_XA} x_{ij} = {\frac{{{\sum\limits_{\alpha \in I_{r,m}
{\left\{ {j} \right\}}} {{\left| \left( {{\rm
{\bf A}}^{ k+1} _{\,j\,.} \left( {\check{{\rm {\bf
b}}}_{i\,.}} \right)} \right)\,_{\alpha} ^{\alpha}\right|} }
}}}{{{\sum\limits_{\alpha \in I_{r,\,m}}  {{\left| {\left( {{\rm {\bf A}}^{ k+1} } \right) {\kern 1pt} _{\alpha}
^{\alpha} } \right|}}} }}}.
\end{equation}
\end{theorem}
\begin{proof}  By Theorem  \ref{theor:dr_repr_row} we can represent the matrix
${\rm {\bf A}}^{ D} $ by (\ref{eq:dr_repr_row}). Therefore, for
all $i=\overline {1,n} $, $j=\overline {1,m} $,  we obtain
\[
x_{ij}  =\sum_{s=1}^{m}b_{is} a_{sj}^{D}=\sum_{s=1}^{m}b_{is}\cdot
{\frac{{{\sum\limits_{\alpha \in I_{r,m} {\left\{ {j} \right\}}}
{{\left| \left( {{\rm {\bf A}}^{ k+1}_{\,j\,.} \left( {\bf a}_{s\,.}^{(k)} \right)}
\right)\,_{\alpha} ^{\alpha}\right|} } }}}{{{\sum\limits_{\alpha
\in I_{r,\,m}}  {{\left| {\left( {{\rm {\bf A}}^{ k+1}
} \right) {\kern 1pt} _{\alpha} ^{\alpha} } \right|}}} }}} =
\]
\[
{\frac{{{\sum_{s=1}^{m}b_{ik} \sum\limits_{\alpha \in I_{r,m}
{\left\{ {j} \right\}}} {{\left| \left( {{\rm
{\bf A}}^{k+1}_{\,j\,.} \left( {\bf a}_{s\,.}^{(k)}
\right)} \right)\,_{\alpha} ^{\alpha}\right|} } } }}
{{{\sum\limits_{\alpha \in I_{r,\,m}}  {{\left| {\left( {{\rm {\bf A}}^{ k+1} } \right) {\kern 1pt} _{\alpha} ^{\alpha} }
\right|}}} }}}
\]
Since for all $i = \overline {1,n} $
 \[{\sum\limits_{s} {{  b_{is}\rm {\bf a}}_{s\,.}^{ (k)}}
}=\begin{pmatrix}
  \sum\limits_{s} {b_{is}a_{s1}^{(k)} } & \sum\limits_{s} {b_{is}a_{s2}^{ (k)} } & \cdots & \sum\limits_{s} {b_{is}a_{sm}^{ (k)} }
\end{pmatrix}
 = \check{{\rm {\bf b}}}_{i.},\]
then it follows (\ref{eq:dr_XA}).
\end{proof}

 Consider a  matrix equation
 \begin{equation}\label{eq:AXB=D}
 {\rm {\bf A}}{\rm {\bf X}}{\rm {\bf B}} = {\rm {\bf
D}},
\end{equation}
where  ${\rm {\bf
A}}\in{\mathbb{C}}^{n\times n}$ with  $Ind{\kern 1pt} {\rm {\bf A}} = k_{1} $,  ${\rm {\bf B}}\in {\mathbb{C}}^{m\times m} $ with  $Ind{\kern 1pt} {\rm {\bf B}} = k_{2} $ and $ {\rm {\bf D}}\in{\rm
{\mathbb{C}}}^{n \times m}$ are given, and  $ {\rm {\bf X}}\in{\rm
{\mathbb{C}}}^{n \times m}$is unknown.
\begin{theorem}(\cite{xu}, Theorem 3) If $R({\bf D})\subset R({\bf A}^{k_{1}})$ and $N({\bf D})\supset N({\bf B}^{k_{2}})$, $k=max\{k_{1}, k_{2}\}$, then the matrix equation (\ref{eq:AXB=D}) with constrain
$R({\bf X})\subset R({\bf A}^{k})$ and $N({\bf X})\supset N({\bf B}^{k})$ has a unique solution
\[{\rm {\bf X}}={\rm {\bf A}}^{D}{\rm {\bf D}}{\rm {\bf B}}^{
D}.\]
\end{theorem}

We denote  ${\rm {\bf A}}^{
 k_{1} }{\rm {\bf D}}{\rm {\bf B}}^{
 k_{2} }=:\widetilde{{\rm {\bf D}}}= (\widetilde{d}_{ij})\in
{\mathbb{C}}^{n\times m}$.
\begin{theorem}\label{theor:AXB=D}
 If  $\rank{\rm {\bf
A}}^{k_{1} + 1} = \rank{\rm {\bf A}}^{k_{1}}=r_{1} \le n$ for ${\rm {\bf A}}\in {\mathbb C}^{n\times n} $, and  $\rank{\rm {\bf
B}}^{k_{2} + 1} = \rank{\rm {\bf B}}^{k_{2}}=r_{2} \le m$ for ${\rm {\bf B}}\in {\mathbb C}^{m\times m} $, then for the
Drazin inverse solution ${\rm {\bf X}}={\rm {\bf A}}^{D}{\rm {\bf D}}{\rm {\bf B}}^{
D}=:(x_{ij})\in
{\mathbb{C}}^{n\times m}$  of (\ref{eq:AXB=D}) we have
\begin{equation}\label{eq:d^B}
x_{ij} = {\frac{{{\sum\limits_{\beta \in J_{r_{1},\,n} {\left\{
{i} \right\}}} { \left| {{\rm {\bf A}}^{ k_{1}+1} _{\,.\,i} \left( {{{\rm {\bf d}}}\,_{.\,j}^{{\rm {\bf B}}}}
\right)\, _{\beta} ^{\beta}} \right| } } }}{{{\sum\limits_{\beta
\in J_{r_{1},n}} {{\left| {\left( {{\rm {\bf A}}^{ k_{1}+1} } \right)_{\beta} ^{\beta} } \right|}} \sum\limits_{\alpha \in
I_{r_{2},m}}{{\left| {\left( {{\rm {\bf B}}^{k_{2}+1} }
\right) _{\alpha} ^{\alpha} } \right|}}} }}},
\end{equation}
or
\begin{equation}\label{eq:d^A}
 x_{ij}={\frac{{{\sum\limits_{\alpha
\in I_{r_{2},m} {\left\{ {j} \right\}}} { \left| {{\rm {\bf B}}^{k_{2}+1} _{j\,.} \left( {{{\rm {\bf
d}}}\,_{i\,.}^{{\rm {\bf A}}}} \right)\,_{\alpha} ^{\alpha}}
\right| } }}}{{{\sum\limits_{\beta \in J_{r_{1},n}} {{\left|
{\left( {{\rm {\bf A}}^{k_{1}+1} } \right) _{\beta}
^{\beta} } \right|}}\sum\limits_{\alpha \in I_{r_{2},m}} {{\left|
{\left( {{\rm {\bf B}}^{k_{2}+1} } \right) _{\alpha}
^{\alpha} } \right|}}} }}},
\end{equation}
where
  \begin{equation}
\label{eq:def_d^B_m}
   {{\rm {\bf d}}_{.\,j}^{{\rm {\bf B}}}}=\left[
\sum\limits_{\alpha \in I_{r_{2},m} {\left\{ {j} \right\}}} {
\left| {{\rm {\bf B}}^{ k_{2}+1} _{j.}
\left( {\widetilde{{\rm {\bf d}}}_{1.}} \right)\,_{\alpha} ^{\alpha}}
\right|},...,\sum\limits_{\alpha \in I_{r_{2},m} {\left\{ {j}
\right\}}} { \left| {{\rm {\bf B}}^{k_{2}+1}_{j.} \left( {\widetilde{{\rm {\bf d}}}_{n.}}
\right)\,_{\alpha} ^{\alpha}} \right|} \right]^{T},
\end{equation}
  \[
  {{\rm {\bf d}}_{i\,.}^{{\rm {\bf A}}}}=\left[
\sum\limits_{\beta \in J_{r_{1},n} {\left\{ {i} \right\}}} {
\left| {{\rm {\bf A}}^{k_{1}+1}_{.i}
\left( {\widetilde{{\rm {\bf d}}}_{.1}} \right)\,_{\beta} ^{\beta}}
\right|},...,\sum\limits_{\alpha \in I_{r_{1},n} {\left\{ {i}
\right\}}} { \left| {{\rm {\bf A}}^{k_{1}+1}_{.i} \left( {\widetilde{{\rm {\bf d}}}_{.\,m}}
\right)\,_{\beta} ^{\beta}} \right|} \right]
\]
 are the column-vector and the row-vector. ${\widetilde{{\rm {\bf d}}}_{i.}}$  and ${\widetilde{{\rm {\bf d}}}_{.j}}$ are respectively the $i$th row  and the $j$th column of $\widetilde{{\rm {\bf D}}}$ for all $i=\overline {1,n}$, $j=\overline {1,m}$.

\end{theorem}
\begin{proof}
By Theorems \ref{theor:dr_repr_col} and
\ref{theor:dr_repr_row} the Drazin inverses ${\rm {\bf
A}}^{D} = \left( {a_{ij}^{D} } \right) \in {\rm
{\mathbb{C}}}_{}^{n\times n} $ and ${\rm {\bf B}}^{D} = \left(
{b_{ij}^{D} } \right) \in {\rm {\mathbb{C}}}^{m\times m} $
possess the following determinantal representations, respectively,
\[
 a_{ij}^{D}  = {\frac{{{\sum\limits_{\beta
\in J_{r_{1},\,n} {\left\{ {i} \right\}}} { \left| {{\rm
{\bf A}}^{k_{1}+1} _{\,. \,i} \left( {{\rm {\bf
a}}_{.j}^{ (k_{1})} }  \right){\kern 1pt}  _{\beta} ^{\beta}} \right| }
} }}{{{\sum\limits_{\beta \in J_{r_{1},\,n}} {{\left| {\left(
{{\rm {\bf A}}^{{k_{1}+1}} } \right){\kern 1pt} _{\beta}
^{\beta} }  \right|}}} }}},
\]
\begin{equation}\label{eq:b+}
 b_{ij}^{D}  =
{\frac{{{\sum\limits_{\alpha \in I_{r_{2},m} {\left\{ {j}
\right\}}} { \left| {{\rm {\bf B}}^{k_{2}+1} _{j\,.\,}
({\rm {\bf b}}_{i.\,}^{ (k_{2})} )\,_{\alpha} ^{\alpha}} \right| }
}}}{{{\sum\limits_{\alpha \in I_{r_{2},m}}  {{\left| {\left( {{\rm {\bf B}}^{k_{2}+1} } \right){\kern 1pt} _{\alpha}
^{\alpha} } \right|}}} }}}.
\end{equation}
Then an entry of the Drazin inverse solution ${\rm {\bf X}}={\rm {\bf A}}^{D}{\rm {\bf D}}{\rm {\bf B}}^{
D}=:(x_{ij})\in
{\mathbb{C}}^{n\times m}$ is

\begin{equation}
\label{eq:sum+} x_{ij} = {{\sum\limits_{s = 1}^{m} {\left(
{{\sum\limits_{t = 1}^{n} {{a}_{it}^{D} d_{ts}} } } \right)}}
{b}_{sj}^{D}}.
\end{equation}
Denote  by $\hat{{\rm {\bf d}}_{.s}}$ the $s$th column of ${\rm
{\bf A}}^{ k}{\rm {\bf D}}=:\hat{{\rm {\bf D}}}=
(\hat{d}_{ij})\in {\mathbb{C}}^{n\times m}$ for all $s=\overline
{1,m}$. It follows from ${\sum\limits_{t} { {\rm {\bf a}}_{.\,t}^{
D}}d_{ts} }=\hat{{\rm {\bf d}}_{.\,s}}$ that
\[
\sum\limits_{t = 1}^{n} {{a}_{it}^{D} d_{ts}}=\sum\limits_{t =
1}^{n}{\frac{{{\sum\limits_{\beta \in J_{r_{1},\,n} {\left\{ {i}
\right\}}} { \left| {{\rm {\bf A}}^{k_{1}+1} _{\,. \,i} \left( {{\rm {\bf a}}_{.t}^{
(k_{1})} } \right) {\kern
1pt} _{\beta} ^{\beta}} \right| } } }}{{{\sum\limits_{\beta \in
J_{r_{1},\,n}} {{\left| {\left( {{\rm {\bf A}}^{k_{1}+1} }
\right){\kern 1pt} _{\beta} ^{\beta} }  \right|}}} }}}\cdot
d_{ts}=
\]
\begin{equation}\label{eq:sum_cdet}
{\frac{{{\sum\limits_{\beta \in J_{r_{1},\,n} {\left\{ {i}
\right\}}}\sum\limits_{t = 1}^{n} { \left| {{\rm {\bf
A}}^{k_{1}+1}_{\,. \,i} \left( {{\rm {\bf
a}}_{.t}^{(k_{1})} } \right) {\kern 1pt} _{\beta} ^{\beta}} \right| } }
}\cdot d_{ts}}{{{\sum\limits_{\beta \in J_{r_{1},\,n}} {{\left|
{\left( {{\rm {\bf A}}^{k_{1}+1}} \right){\kern 1pt}
_{\beta} ^{\beta} }  \right|}}} }}}={\frac{{{\sum\limits_{\beta
\in J_{r_{1},\,n} {\left\{ {i} \right\}}} { \left| {{\rm
{\bf A}}^{k_{1}+1}_{\,. \,i} \left( \hat{{\rm
{\bf d}}_{.\,s}} \right) {\kern 1pt} _{\beta} ^{\beta}} \right| }
} }}{{{\sum\limits_{\beta \in J_{r_{1},\,n}} {{\left| {\left(
{{\rm {\bf A}}^{k_{1}+1}} \right){\kern 1pt} _{\beta}
^{\beta} }  \right|}}} }}}
\end{equation}

Substituting  (\ref{eq:sum_cdet}) and (\ref{eq:b+}) in
(\ref{eq:sum+}), we obtain
\[
x_{ij} =\sum\limits_{s = 1}^{m}{\frac{{{\sum\limits_{\beta \in
J_{r_{1},\,n} {\left\{ {i} \right\}}} { \left| {{\rm {\bf
A}}^{k_{1}+1} _{\,. \,i} \left( \hat{{\rm {\bf
d}}_{.\,s}} \right) {\kern 1pt} _{\beta} ^{\beta}} \right| } }
}}{{{\sum\limits_{\beta \in J_{r_{1},\,n}} {{\left| {\left( {{\rm
{\bf A}}^{k_{1}+1}} \right){\kern 1pt} _{\beta} ^{\beta}
}  \right|}}} }}}{\frac{{{\sum\limits_{\alpha \in I_{r_{2},m}
{\left\{ {j} \right\}}} { \left| {{\rm {\bf B}}^{k_{2}+1}
_{j\,.\,} ({\rm {\bf b}}_{s.\,}^{ (k_{2})} )\,_{\alpha} ^{\alpha}}
\right| } }}}{{{\sum\limits_{\alpha \in I_{r_{2},m}}  {{\left|
{\left( {{\rm {\bf B}}^{k_{2}+1} } \right){\kern 1pt}
_{\alpha} ^{\alpha} } \right|}}} }}}.
\]
Suppose ${\rm {\bf e}}_{s.}$ and ${\rm {\bf e}}_{.\,s}$ are
respectively the unit row-vector and the unit column-vector whose
components are $0$, except the $s$th components, which are $1$. Since
\[
\hat{{\rm {\bf d}}_{.\,s}}=\sum\limits_{l = 1}^{n}{\rm {\bf
e}}_{.\,l}\hat{ d_{ls}},\,\,\,\,  {\rm {\bf b}}_{s.\,}^{
(k_{2})}=\sum\limits_{t = 1}^{m}b_{st}^{(k_{2})}{\rm {\bf e}}_{t.},\,\,\,\,
\sum\limits_{s=1}^{m}\hat{d_{ls}}b_{st}^{(k_{2})}=\widetilde{d}_{lt},
\]
then we have
\[
x_{ij} = \]
\[{\frac{{ \sum\limits_{s = 1}^{m}\sum\limits_{t =
1}^{m} \sum\limits_{l = 1}^{n} {\sum\limits_{\beta \in
J_{r_{1},\,n} {\left\{ {i} \right\}}} { \left| {{\rm {\bf
A}}^{k_{1}+1}_{\,. \,i} \left( {\rm {\bf
e}}_{.\,l} \right) {\kern 1pt} _{\beta} ^{\beta}} \right| } }
}\hat{ d_{ls}}b_{st}^{(k_{2})}{\sum\limits_{\alpha \in I_{r_{2},m}
{\left\{ {j} \right\}}} { \left| {{\rm {\bf B}}^{k_{2}+1}
_{j\,.\,} ({\rm {\bf e}}_{t.} )\,_{\alpha} ^{\alpha}} \right| } }
}{{{\sum\limits_{\beta \in J_{r_{1},\,n}} {{\left| {\left( {{\rm
{\bf A}}^{k_{1}+1}} \right){\kern 1pt} _{\beta} ^{\beta}
}  \right|}}} }{{\sum\limits_{\alpha \in I_{r_{2},m}} {{\left|
{\left( {{\rm {\bf B}}^{k_{2}+1} } \right){\kern 1pt}
_{\alpha} ^{\alpha} } \right|}}} }}    }=
\]
\begin{equation}\label{eq:x_ij}
{\frac{{ \sum\limits_{t = 1}^{m} \sum\limits_{l = 1}^{n}
{\sum\limits_{\beta \in J_{r_{1},\,n} {\left\{ {i} \right\}}}
{\left| {{\rm {\bf A}}^{k_{1}+1}_{\,.
\,i} \left( {\rm {\bf e}}_{.\,l} \right) {\kern 1pt} _{\beta}
^{\beta}} \right| } } }\,\,\widetilde{d}_{lt}{\sum\limits_{\alpha
\in I_{r_{2},m} {\left\{ {j} \right\}}} { \left| {{\rm {\bf B}}^{k_{2}+1}_{j\,.\,} ({\rm {\bf e}}_{t.} )\,_{\alpha}
^{\alpha}} \right| } } }{{{\sum\limits_{\beta \in J_{r_{1},\,n}}
{{\left| {\left( {{\rm {\bf A}}^{k_{1}+1}} \right){\kern
1pt} _{\beta} ^{\beta} }  \right|}}} }{{\sum\limits_{\alpha \in
I_{r_{2},m}} {{\left| {\left( {{\rm {\bf B}}^{k_{2}+1} }
\right){\kern 1pt} _{\alpha} ^{\alpha} } \right|}}} }}    }.
\end{equation}
Denote by
\[
 d^{{\rm {\bf A}}}_{it}:= \]
\[
{\sum\limits_{\beta \in J_{r_{1},\,n} {\left\{ {i} \right\}}} {
\left| {{\rm {\bf A}}^{k_{1}+1}_{\,.
\,i} \left( \widetilde{{\rm {\bf d}}}_{.\,t} \right){\kern 1pt}
_{\beta} ^{\beta}} \right| } }= \sum\limits_{l = 1}^{n}
{\sum\limits_{\beta \in J_{r_{1},\,n} {\left\{ {i} \right\}}} {
\left| {{\rm {\bf A}}^{k_{1}+1}_{\,.
\,i} \left( {\rm {\bf e}}_{.\,l} \right){\kern 1pt}
 _{\beta} ^{\beta}} \right| } } \widetilde{d}_{lt}
\]
the $t$th component  of a row-vector ${\rm {\bf d}}^{{\rm {\bf
A}}}_{i\,.}= (d^{{\rm {\bf A}}}_{i1},...,d^{{\rm {\bf A}}}_{im})$
for all $t=\overline {1,m}$. Substituting it in (\ref{eq:x_ij}),
we obtain
\[x_{ij} ={\frac{{ \sum\limits_{t = 1}^{m}
 d^{{\rm {\bf A}}}_{it}
}{\sum\limits_{\alpha \in I_{r_{2},m} {\left\{ {j} \right\}}} {
\left| {{\rm {\bf B}}^{k_{2}+1} _{j\,.\,} ({\rm {\bf
e}}_{t.} )\,_{\alpha} ^{\alpha}} \right| } }
}{{{\sum\limits_{\beta \in J_{r_{1},\,n}} {{\left| {\left( {{\rm
{\bf A}}^{k_{1}+1}} \right){\kern 1pt} _{\beta} ^{\beta}
}  \right|}}} }{{\sum\limits_{\alpha \in I_{r_{2},m}} {{\left|
{\left( {{\rm {\bf B}}^{k_{2}+1} } \right){\kern 1pt}
_{\alpha} ^{\alpha} } \right|}}} }}    }.
\]
Since $\sum\limits_{t = 1}^{m}
 d^{{\rm {\bf A}}}_{it}{\rm {\bf e}}_{t.}={\rm {\bf
d}}^{{\rm {\bf A}}}_{i\,.}$, then it follows (\ref{eq:d^A}).

If we denote by
\[
 d^{{\rm {\bf B}}}_{lj}:=
\sum\limits_{t = 1}^{m}\widetilde{d}_{lt}{\sum\limits_{\alpha \in
I_{r_{2},m} {\left\{ {j} \right\}}} { \left| {{\rm
{\bf B}}^{k_{2}+1} _{j\,.\,} ({\rm {\bf e}}_{t.} )\,_{\alpha}
^{\alpha}} \right| } }={\sum\limits_{\alpha \in I_{r_{2},m}
{\left\{ {j} \right\}}} {\left| {{\rm {\bf B}}^{k_{2}+1}
_{j\,.\,} (\widetilde{{\rm {\bf d}}}_{l.} )\,_{\alpha} ^{\alpha}}
\right| } }
\]
the $l$th component  of a column-vector ${\rm {\bf d}}^{{\rm {\bf
B}}}_{.\,j}= (d^{{\rm {\bf B}}}_{1j},...,d^{{\rm {\bf
B}}}_{jn})^{T}$ for all $l=\overline {1,n}$ and substitute it in
(\ref{eq:x_ij}), we obtain
\[x_{ij} ={\frac{{  \sum\limits_{l = 1}^{n}
{\sum\limits_{\beta \in J_{r_{1},\,n} {\left\{ {i} \right\}}} {
\left| {{\rm {\bf A}}^{k_{1}+1}_{\,.
\,i} \left( {\rm {\bf e}}_{.\,l} \right){\kern 1pt} _{\beta}
^{\beta}} \right| } } }\,\,d^{{\rm {\bf B}}}_{lj}
}{{{\sum\limits_{\beta \in J_{r_{1},\,n}} {{\left| {\left( {{\rm
{\bf A}}^{k_{1}+1}} \right){\kern 1pt} _{\beta} ^{\beta}
}  \right|}}} }{{\sum\limits_{\alpha \in I_{r_{2},m}} {{\left|
{\left( {{\rm {\bf B}}^{k_{2}+1} } \right){\kern 1pt}
_{\alpha} ^{\alpha} } \right|}}} }}    }.
\]
Since $\sum\limits_{l = 1}^{n}{\rm {\bf e}}_{.l}
 d^{{\rm {\bf B}}}_{lj}={\rm {\bf
d}}^{{\rm {\bf B}}}_{.\,j}$, then it follows (\ref{eq:d^B}).

\end{proof}

\section{Applications of the determinantal representations of the Drazin inverse to some differential matrix equations}
Consider the
 matrix differential equation
 \begin{equation}\label{eq:left_dif}
 {\bf X}'+ {\bf A}{\bf X}={\bf B}
\end{equation}
where $ {\bf
A}\in{\rm {\mathbb{C}}}^{n \times n}$, $ {\bf B}\in{\rm
{\mathbb{C}}}^{n \times n}$ are given, $ {\rm {\bf X}}\in{\rm
{\mathbb{C}}}^{n \times n}$ is unknown. It's well-known that the general solution of
 (\ref{eq:left_dif})
is found to be
 \[
 {\bf X}(t)=\exp^{ -{\bf A}t}\left(\int \exp^{ {\bf A}t}dt\right){\bf B}
\]
If ${\bf A}$ is invertible, then
\[\int \exp^{ {\bf A}t}dt={\bf A}^{-1}\exp^{ {\bf A}t}+{\bf G},
\]
where ${\bf G}$ is an arbitrary $n\times n$ matrix. If ${\bf A}$ is singular, then the following theorem gives an answer.
 \begin{theorem}(\cite{ca1}, Theorem 1) \label{theor:int_to_draz}If ${\bf A}$ has index $k$, then
 \[
 \int \exp^{ {\bf A}t}dt={\bf A}^{D}\exp^{ {\bf A}t}+({\bf I}-{\bf A}{\bf A}^{D})t\left[{\bf I}+\frac{{\bf A}}{2}t+\frac{{\bf A}^{2}}{3!}t^{2}+...+\frac{{\bf A}^{k-1}}{k!}t^{k-1}\right]+{\bf G}. \]

\end{theorem}
Using Theorem \ref{theor:int_to_draz} and the power series expansion
of $\exp^{ -{\bf A}t}$, we get an explicit form for a general solution of (\ref{eq:left_dif})
 \[\begin{array}{l}
     {\bf X}(t)= \\
      \left\{{\bf A}^{D}+({\bf I}-{\bf A}{\bf A}^{D})t\left({\bf I}-\frac{{\bf A}}{2}t+\frac{{\bf A}^{2}}{3!}t^{2}-...(-1)^{k-1}\frac{{\bf A}^{k-1}}{k!}t^{k-1}\right)+{\bf G}\right\}{\bf B}.
   \end{array}
 \]
If we put ${ \bf G}={\bf 0}$, then we obtain the following  partial solution of (\ref{eq:left_dif}),

 \begin{equation}\label{eq:left_dif_part_sol}\begin{array}{c}
                                               {\bf X}(t)={\bf A}^{D}{\bf B}+({\bf B}-{\bf A}^{D}{\bf A}{\bf B})t-\frac{1}{2}({\bf A}{\bf B}-{\bf A}^{D}{\bf A}^{2}{\bf B})t^{2}+... \\
                                                \frac{(-1)^{k-1}}{k!}({\bf A}^{k-1}{\bf B}-{\bf A}^{D}{\bf A}^{k}{\bf B})t^{k}.
                                             \end{array}
\end{equation}

Denote ${\rm {\bf A}}^{l}{\rm {\bf B}}=:\widehat{{{\rm
{\bf B}}}}^{(l)}
= (\widehat{b}^{(l)}_{ij})\in {\mathbb{C}}^{n\times n}$ for all $l=\overline {1,2k}$.

 \begin{theorem} The partial solution (\ref{eq:left_dif_part_sol}), ${\bf X}(t)=(x_{ij})$, possess the following determinantal representation,
  \begin{equation}\label{eq:left_dif_repr}
\begin{array}{l}
   x_{ij}=
  {\frac{{{\sum\limits_{\beta \in
J_{r,\,n} {\left\{ {i} \right\}}} {{\left| \left( {{\rm
{\bf A}}^{k+1}_{\,.\,i} \left( {\widehat{{\rm
{\bf b}}}^{(k)}_{.j}} \right)} \right){\kern 1pt} {\kern 1pt} _{\beta}
^{\beta} \right|}} } }}{{{\sum\limits_{\beta \in J_{r,\,\,n}}
{{\left| {\left( {{\rm {\bf A}}^{ k+1} } \right){\kern
1pt} {\kern 1pt} _{\beta} ^{\beta} }  \right|}}} }}}+
 \left ({b_{ij}- {\frac{{{\sum\limits_{\beta \in
J_{r,\,n} {\left\{ {i} \right\}}} {{\left| \left( {{\rm
{\bf A}}^{k+1}_{\,.\,i} \left( {\widehat{{\rm
{\bf b}}}^{(k+1)}_{.j}} \right)} \right){\kern 1pt} {\kern 1pt} _{\beta}
^{\beta} \right|}} } }}{{{\sum\limits_{\beta \in J_{r,\,\,n}}
{{\left| {\left( {{\rm {\bf A}}^{k+1} } \right){\kern
1pt} {\kern 1pt} _{\beta} ^{\beta} }  \right|}}} }}}} \right)t\\
-\frac{1}{2} \left ({\widehat{b}^{(1)}_{ij}- {\frac{{{\sum\limits_{\beta \in
J_{r,\,n} {\left\{ {i} \right\}}} {{\left| \left( {{\rm
{\bf A}}^{k+1}_{\,.\,i} \left( {\widehat{{\rm
{\bf b}}}^{(k+2)}_{.j}} \right)} \right){\kern 1pt} {\kern 1pt} _{\beta}
^{\beta} \right|}} } }}{{{\sum\limits_{\beta \in J_{r,\,\,n}}
{{\left| {\left( {{\rm {\bf A}}^{k+1} } \right){\kern
1pt} {\kern 1pt} _{\beta} ^{\beta} }  \right|}}} }}}} \right)t^{2}+...\\\frac{(-1)^{k}}{k!}
\left ({\widehat{b}^{(k-1)}_{ij}- {\frac{{{\sum\limits_{\beta \in
J_{r,\,n} {\left\{ {i} \right\}}} {{\left| \left( {{\rm
{\bf A}}^{k+1}_{\,.\,i} \left( {\widehat{{\rm
{\bf b}}}^{(2k)}_{.j}} \right)} \right){\kern 1pt} {\kern 1pt} _{\beta}
^{\beta} \right|}} } }}{{{\sum\limits_{\beta \in J_{r,\,\,n}}
{{\left| {\left( {{\rm {\bf A}}^{k+1} } \right){\kern
1pt} {\kern 1pt} _{\beta} ^{\beta} }  \right|}}} }}}} \right)t^{k}
\end{array}
\end{equation}
for all $i,j=\overline {1,n}$.

\end{theorem}
\begin{proof}
Using the determinantal representation of the  identity  (\ref{eq:AdA}) we obtain the following determinantal representation of the matrix ${\bf A}^{D}{\bf A}^{m}{\bf B}:=(y_{ij})$,

\[y_{ij}=\sum\limits_{s = 1}^{n}p_{is}\sum\limits_{t = 1}^{n}a^{(m-1)}_{st}b_{tj}=
{\sum\limits_{\beta \in J_{r,n} {\left\{ {i} \right\}}}}\frac{{ \sum\limits_{s = 1}^{n}{{\left|
{\left( {{\rm {\bf A}}_{.\,i}^{k+1} \left( {{\rm {\bf
a}}_{.s}}^{(k+1)} \right)} \right)_{\beta} ^{\beta} } \right|}}}\cdot \sum\limits_{t = 1}^{n}a^{(m-1)}_{st}b_{tj}}{{\sum\limits_{\beta \in J_{r,n} } {{\left|
{\left( {{\rm {\bf A}}^{k+1} } \right)_{\beta} ^{\beta} } \right|}}}}=
\]
\[{\sum\limits_{\beta \in J_{r,n} {\left\{ {i} \right\}}}}\frac{{ \sum\limits_{t= 1}^{n}{{\left|
{\left( {{\rm {\bf A}}_{.\,i}^{k+1} \left( {{\rm {\bf
a}}_{.t}}^{(k+m)} \right)} \right)_{\beta} ^{\beta} } \right|}}}\cdot b_{tj}}{{\sum\limits_{\beta \in J_{r,n} } {{\left|
{\left( {{\rm {\bf A}}^{k+1} } \right)_{\beta} ^{\beta} } \right|}}}}=
  {\frac{{{\sum\limits_{\beta \in
J_{r,\,n} {\left\{ {i} \right\}}} {{\left| \left( {{\rm
{\bf A}}^{k+1}_{\,.\,i} \left( {\widehat{{\rm
{\bf b}}}^{(k+m)}_{.j}} \right)} \right){\kern 1pt} {\kern 1pt} _{\beta}
^{\beta} \right|}} } }}{{{\sum\limits_{\beta \in J_{r,\,\,n}}
{{\left| {\left( {{\rm {\bf A}}^{ k+1} } \right){\kern
1pt} {\kern 1pt} _{\beta} ^{\beta} }  \right|}}} }}}
\]
for all $i,j=\overline {1,n}$ and $m=\overline {1,k}$.
 From this and the determinantal representation of the Drazin inverse solution (\ref{eq:dr_AX}) and the identity  (\ref{eq:AdA}) it follows (\ref{eq:left_dif_repr}).
\end{proof}
\begin{corollary}  If $Ind {\bf A}=1$, then the partial solution of (\ref{eq:left_dif}),
\[ {\bf X}(t)=(x_{ij})={\bf A}^{g}{\bf B}+({\bf B}-{\bf A}^{g}{\bf A}{\bf B})t,\]
possess the following determinantal representation
 \begin{equation}\label{eq:left_dif_repr_gr}
x_{ij}= {\frac{{{\sum\limits_{\beta \in
J_{r,\,n} {\left\{ {i} \right\}}} {{\left| \left( {{\rm
{\bf A}}^{2}_{\,.\,i} \left( {\widehat{{\rm
{\bf b}}}_{.j}}^{(1)} \right)} \right){\kern 1pt} {\kern 1pt} _{\beta}
^{\beta} \right|}} } }}{{{\sum\limits_{\beta \in J_{r,\,\,n}}
{{\left| {\left( {{\rm {\bf A}}^{ 2} } \right){\kern
1pt} {\kern 1pt} _{\beta} ^{\beta} }  \right|}}} }}}+
 \left ({b_{ij}- {\frac{{{\sum\limits_{\beta \in
J_{r,\,n} {\left\{ {i} \right\}}} {{\left| \left( {{\rm
{\bf A}}^{2}_{\,.\,i} \left( {{\widehat{{\rm
{\bf b}}}}_{.j}}^{(2)} \right)} \right){\kern 1pt} {\kern 1pt} _{\beta}
^{\beta} \right|}} } }}{{{\sum\limits_{\beta \in J_{r,\,\,n}}
{{\left| {\left( {{\rm {\bf A}}^{ 2} } \right){\kern
1pt} {\kern 1pt} _{\beta} ^{\beta} }  \right|}}} }}}} \right)t.
\end{equation}
for all $i,j=\overline {1,n}$.
\end{corollary}

Consider the
 matrix differential equation
 \begin{equation}\label{eq:right_dif}
 {\bf X}'+ {\bf X}{\bf A}={\bf B}
\end{equation}
where $ {\bf
A}\in{\rm {\mathbb{C}}}^{n \times n}$, $ {\bf B}\in{\rm
{\mathbb{C}}}^{n \times n}$ are given, $ {\rm {\bf X}}\in{\rm
{\mathbb{C}}}^{n \times n}$ is unknown. The general solution of
 (\ref{eq:right_dif})
is found to be
 \[
 {\bf X}(t)={\bf B}\exp^{ -{\bf A}t}\left(\int \exp^{ {\bf A}t}dt\right)
\]
If ${\bf A}$ is singular, then an explicit form for a general solution of (\ref{eq:right_dif}) is
 \[\begin{array}{l}
     {\bf X}(t)=\\
     {\bf B}\left\{{\bf A}^{D}+({\bf I}-{\bf A}{\bf A}^{D})t\left({\bf I}-\frac{{\bf A}}{2}t+\frac{{\bf A}^{2}}{3!}t^{2}+...(-1)^{k-1}\frac{{\bf A}^{k-1}}{k!}t^{k-1}\right)+{\bf G}\right\}.
   \end{array}
 \]
If we put ${ \bf G}={\bf 0}$, then we obtain the following  partial solution of (\ref{eq:right_dif}),
  \begin{equation}\label{eq:right_dif_part_sol}\begin{array}{c}
                                                 {\bf X}(t)={\bf B}{\bf A}^{D}+({\bf B}-{\bf B}{\bf A}{\bf A}^{D})t-\frac{1}{2}({\bf B}{\bf A}-{\bf B}{\bf A}^{2}{\bf A}^{D})t^{2}+... \\
                                                  \frac{(-1)^{k-1}}{k!}({\bf B}{\bf A}^{k-1}-{\bf B}{\bf A}^{k}{\bf A}^{D})t^{k}.
                                               \end{array}
\end{equation}
Denote ${\rm {\bf B}}{\rm {\bf A}}^{l}=:\check{{{\rm
{\bf B}}}}^{(l)}
= (\check{b}^{(l)}_{ij})\in {\mathbb{C}}^{n\times n}$ for all $l=\overline {1,2k}$.
Using the determinantal representation of the Drazin inverse solution (\ref{eq:dr_XA}), the group inverse (\ref{eq:AAg}) and the identity  (\ref{eq:AAd})
we evidently obtain the following theorem.
 \begin{theorem} The partial solution (\ref{eq:right_dif_part_sol}), ${\bf X}(t)=(x_{ij})$, possess the following determinantal representation,
  \[
\begin{array}{l}
   x_{ij}=
  {\frac{{{\sum\limits_{\alpha \in I_{r,n}
{\left\{ {j} \right\}}} {{\left| \left( {{\rm
{\bf A}}^{k+1}_{j\,.} \left( {\check{{\rm
{\bf b}}}^{(k)}_{.\,i}} \right)} \right){\kern 1pt} {\kern 1pt} _{\alpha} ^{\alpha}  \right|}} } }}{{{\sum\limits_{\alpha \in I_{r,n} }
{{\left| {\left( {{\rm {\bf A}}^{ k+1} } \right){\kern
1pt} {\kern 1pt} _{\alpha} ^{\alpha}  }  \right|}}} }}}+
 \left ({b_{ij}- {\frac{{{\sum\limits_{\alpha \in I_{r,n}
{\left\{ {j} \right\}}} {{\left| \left( {{\rm
{\bf A}}^{k+1}_{j\,.} \left( {\check{{\rm
{\bf b}}}^{(k+1)}_{i\,.}} \right)} \right){\kern 1pt} {\kern 1pt} _{\alpha} ^{\alpha}  \right|}} } }}{{{\sum\limits_{\alpha \in I_{r,n} } {{\left| {\left( {{\rm {\bf A}}^{k+1} } \right){\kern
1pt} {\kern 1pt} _{\alpha} ^{\alpha}  }  \right|}}} }}}} \right)t \\
-\frac{1}{2} \left ({\check{b}^{(1)}_{ij}- {\frac{{{\sum\limits_{\alpha \in I_{r,n}
{\left\{ {j} \right\}}} {{\left| \left( {{\rm
{\bf A}}^{k+1}_{j\,.} \left( {\check{{\rm
{\bf b}}}^{(k+2)}_{i\,.}} \right)} \right){\kern 1pt} {\kern 1pt} _{\alpha} ^{\alpha}  \right|}} } }}{{{\sum\limits_{\alpha \in I_{r,n} } {{\left| {\left( {{\rm {\bf A}}^{k+1} } \right){\kern
1pt} {\kern 1pt} _{\alpha} ^{\alpha}  }  \right|}}} }}}} \right)t^{2}+...\\\frac{(-1)^{k}}{k!}
\left ({\check{b}^{(k-1)}_{ij}- {\frac{{{\sum\limits_{\alpha \in I_{r,n}
{\left\{ {j} \right\}}} {{\left| \left( {{\rm
{\bf A}}^{k+1}_{j\,.} \left( {\check{{\rm
{\bf b}}}^{(2k)}_{i\,.}} \right)} \right){\kern 1pt} {\kern 1pt} _{\alpha} ^{\alpha} \right|}} } }}{{{\sum\limits_{\alpha \in I_{r,n} }
{{\left| {\left( {{\rm {\bf A}}^{k+1} } \right){\kern
1pt} {\kern 1pt} _{\alpha} ^{\alpha} }  \right|}}} }}}} \right)t^{k}
\end{array}
\]
for all $i,j=\overline {1,n}$.
\end{theorem}
\begin{corollary}  If $Ind {\bf A}=1$, then the partial solution of (\ref{eq:right_dif}),
\[ {\bf X}(t)=(x_{ij})={\bf B}{\bf A}^{g}+({\bf B}-{\bf B}{\bf A}{\bf A}^{g})t,\]
possess the following determinantal representation
 \[
x_{ij}= {\frac{{{\sum\limits_{\alpha \in I_{r,n} {\left\{ {j} \right\}}} {{\left| \left( {{\rm
{\bf A}}^{2}_{j\,.} \left( {\widehat{{\rm
{\bf b}}}_{i\,.}}^{(1)}  \right)} \right){\kern 1pt} {\kern 1pt} _{\alpha}
^{\alpha} \right|}} } }}{{{\sum\limits_{\alpha \in I_{r,n}}
{{\left| {\left( {{\rm {\bf A}}^{ 2} } \right){\kern
1pt} {\kern 1pt} _{\alpha} ^{\alpha} }  \right|}}} }}}+
 \left ({b_{ij}- {\frac{{{\sum\limits_{\alpha \in I_{r,n} {\left\{ {j} \right\}}} {{\left| \left( {{\rm
{\bf A}}^{2}_{j\,.} \left( {{\widehat{{\rm
{\bf b}}}}_{i\,.}}^{(2)} \right)} \right){\kern 1pt} {\kern 1pt} _{\alpha}
^{\alpha} \right|}} } }}{{{\sum\limits_{\alpha  \in I_{r,n}}
{{\left| {\left( {{\rm {\bf A}}^{ 2} } \right){\kern
1pt} {\kern 1pt} _{\alpha}
^{\alpha} }  \right|}}} }}}} \right)t.
\]
for all $i,j=\overline {1,n}$.
\end{corollary}

\section{Examples}
In this section, we give  examples to illustrate our results.
\subsection{Example 1}
 Let
us consider the matrix equation
\begin{equation}\label{eq_ex:AXB=D}
 {\rm {\bf A}}{\rm {\bf X}}{\rm {\bf B}} = {\rm {\bf
D}},
\end{equation}
where
\[{\bf A}=\begin{pmatrix}
  2 & 0 & 0 \\
  -i & i & i \\
  -i & -i & -i
\end{pmatrix},\,\, {\bf B}=\begin{pmatrix}
  1 & -1 & 1 \\
  i & -i & i \\
  -1 & 1 & 2
\end{pmatrix},\,\, {\bf D}=\begin{pmatrix}
  1 & i & 1 \\
  i & 0 & 1 \\
  1 & i & 0
\end{pmatrix}.\]
We shall find the Drazin inverse solution of (\ref{eq_ex:AXB=D}) by
(\ref{eq:d^B}). We obtain
\[{\bf A}^{2}=\begin{pmatrix}
  4 & 0 & 0 \\
  2-2i & 0 & 0 \\
  -2-2i & 0 & 0
\end{pmatrix},\,{\bf A}^{3}=\begin{pmatrix}
  8 & 0 & 0 \\
  4-4i & 0 & 0 \\
  -4-4i & 0 & 0
\end{pmatrix},\] \[ {\bf B}^{2}=\begin{pmatrix}
  -i & i & 3-i \\
  1 & -1 & 1+3i \\
  -3+i & 3-i & 3+i
\end{pmatrix}.\]
Since ${\rm rank}\,{\rm {\bf A}} = 2$ and ${\rm rank}\,{\rm {\bf A}}^{2}= {\rm rank}\,{\rm {\bf A}}^{2} = 1$, then $k_{1}={\rm Ind}\,{\rm {\bf A}}=2$ and $r_{1}=1$. Since ${\rm rank}\,{\rm {\bf B}} ={\rm rank}\,{\rm {\bf B}}^{2} = 2$, then $k_{2}={\rm Ind}\,{\rm {\bf B}}=1$ and $r_{2}=2$.
 Then we have
\[ {\rm {\bf \widetilde{D}}}= {\rm {\bf
A}}^{2}{\rm {\bf D}}{\rm {\bf B}}=\begin{pmatrix}
  -4 & 4 & 8\\
  -2+2i & 2-2i & 4-4i\\
  2+2i & -2-2i & -4-4i
  \end{pmatrix},\]
  and ${{{\sum\limits_{\beta \in J_{1,\,3}}
{{\left| {\left( {{\rm {\bf A}}^{ 3} } \right) {\kern
1pt} _{\beta} ^{\beta} } \right|}}} }}=
8+0+0=8,$
\[\begin{array}{l}
    {{{\sum\limits_{\alpha \in I_{2,\,3}}  {{\left| {\left( {{\rm {\bf B}}^{2} } \right) {\kern 1pt} _{\alpha} ^{\alpha} }
\right|}}} }}=\\ \det\begin{pmatrix}
  -i & i \\
  1 & -1
\end{pmatrix}+\det\begin{pmatrix}
  -1 & 1+3i \\
  3-i & 3+i
\end{pmatrix}+\det\begin{pmatrix}
  -i & 3-i \\
  -3+i & 3+i
\end{pmatrix}= \\
    0+(-9-9i)+(9-9i)=-18i.
  \end{array}
\]
By (\ref{eq:def_d^B_m}), we can get \[{\bf d}_{.1}^{{\bf
B}}=\begin{pmatrix}
  12-12i \\
  -12i \\
  -12
\end{pmatrix},\,\,{\bf d}_{.2}^{{\bf  B}}=\begin{pmatrix}
  -12+12i \\
  12i \\
  12
\end{pmatrix},\,\,{\bf d}_{.3}^{{\bf  B}}=\begin{pmatrix}
  8 \\
  -12-12i\\
  -12+12i
\end{pmatrix}.\]
Since ${\rm {\bf A}}^{ 3}_{\,.\,1}
\left( {{{\rm {\bf d}}}\,_{.\,1}^{{\rm {\bf B}}}} \right)=
\begin{pmatrix}
  12-12i & 0 & 0 \\
   -12i & 0 & 0 \\
  -12 & 0 & 0
\end{pmatrix}$, then finally we obtain
\[
x_{11} = {\frac{{{\sum\limits_{\beta \in J_{1,\,3} {\left\{ {1}
\right\}}} { \left| {{\rm {\bf A}}^{ 3}_{\,.\,1} \left( {{{\rm {\bf d}}}\,_{.\,1}^{{\rm {\bf B}}}}
\right)\, _{\beta} ^{\beta}} \right| } } }}{{{\sum\limits_{\beta
\in J_{1,3}} {{\left| {\left( {{\rm {\bf A}}^{ 3} }
\right)_{\beta} ^{\beta} } \right|}} \sum\limits_{\alpha \in
I_{2,3}}{{\left| {\left( {{\rm {\bf B}}^{2} }
\right) _{\alpha} ^{\alpha} } \right|}}} }}}=\frac{12-12i}{8\cdot (-18i)}=\frac{1+i}{12}.\]
Similarly,
\[
  x_{12} =\frac{-12+12i}{8\cdot (-18i)}=\frac{-1-i}{12},\,\,\,
  x_{13} =\frac{8}{8\cdot (-18i)}=\frac{i}{18},\]

   \[x_{21} =\frac{-12i}{8\cdot (-18i)}=\frac{1}{12},\,\,
   x_{22} =\frac{12i}{8\cdot (-18i)}=-\frac{1}{12},\,\,
  x_{23} =\frac{-12-12i}{8\cdot (-18i)}=\frac{1-i}{12},\]

   \[   x_{31} =\frac{12}{8\cdot (-18i)}=-\frac{i}{12},\,\,
    x_{32} =\frac{-12}{8\cdot (-18i)}=\frac{i}{12}.
\,\,
  x_{33} =\frac{-12+12i}{8\cdot (-18i)}=\frac{-1-i}{12}.\]
  Then
\[{\rm {\bf X}}=\left(
                                    \begin{array}{ccc}
                                      \frac{1+i}{12} & \frac{-1-i}{12} & \frac{i}{18} \\
                                      \frac{1}{12} & -\frac{1}{12} & \frac{1-i}{12} \\
                                      -\frac{i}{12} & \frac{i}{12} & \frac{-1-i}{12}\\
                                    \end{array}
                                  \right)
                                   \]
is the Drazin inverse solution of (\ref{eq_ex:AXB=D}).
\subsection{Example 2}
 Let
us consider the differential matrix equation
\begin{equation}\label{eq_ex:X'=AX+B}
{\bf X}'+ {\rm {\bf A}} {\bf X} = {\rm {\bf
B}},
\end{equation}
where
\[ {\bf A}=\begin{pmatrix}
  1 & -1 & 1 \\
  i & -i & i \\
  -1 & 1 & 2
\end{pmatrix},\,\, {\bf B}=\begin{pmatrix}
  1 & i & 1 \\
  i & 0 & 1 \\
  1 & i & 0
\end{pmatrix}.\]
Since ${\rm rank}\,{\rm {\bf A}} ={\rm rank}\,{\rm {\bf A}}^{2} = 2$, then $k={\rm Ind}\,{\rm {\bf A}}=1$ and $r=2$.
The matrix ${\bf A}$ is the group inverse. We shall find the partial solution of (\ref{eq_ex:X'=AX+B}) by (\ref{eq:left_dif_repr_gr}). We have
\[ {\bf A}^{2}=\begin{pmatrix}
  -i & i & 3-i \\
  1 & -1 & 1+3i \\
  -3+i & 3-i & 3+i
\end{pmatrix},\,\, \widehat{{{\rm
{\bf B}}}}^{(1)}={\bf A}{\bf B}= \begin{pmatrix}
 2 -i & 2i & 0 \\
  1+2i & -2 & 0 \\
  1+i & i & 0
\end{pmatrix},\] \[\widehat{{{\rm
{\bf B}}}}^{(2)}={\bf A}^{2}{\bf B}= \begin{pmatrix}
 2 -2i & 2+3i & 0 \\
  2+2i & -3+2i & 0 \\
  1+5i & -2 & 0
\end{pmatrix}.\]
 and
\[\begin{array}{l}
    {{{\sum\limits_{\alpha \in J_{2,\,3}}  {{\left| {\left( {{\rm {\bf A}}^{2} } \right) {\kern 1pt} _{\beta} ^{\beta} }
\right|}}} }}=\\ \det\begin{pmatrix}
  -i & i \\
  1 & -1
\end{pmatrix}+\det\begin{pmatrix}
  -1 & 1+3i \\
  3-i & 3+i
\end{pmatrix}+\det\begin{pmatrix}
  -i & 3-i \\
  -3+i & 3+i
\end{pmatrix}= \\
    0+(-9-9i)+(9-9i)=-18i.
  \end{array}
\]
Since $\left( {{\rm {\bf A}}^{ 2} } \right)_{\,.\,1}
 \left( {{\widehat{{\rm
{\bf b}}}}_{.1}}^{(1)} \right)=
\begin{pmatrix}
  2-i & i & 3-i \\
   1+2i &-1 & 1+3i \\
 1+i&  3-i & 3+i
\end{pmatrix}$ and \[\left( {{\rm {\bf A}}^{ 2} } \right)_{\,.\,1}
 \left( {{\widehat{{\rm
{\bf b}}}}_{.1}}^{(2)} \right)=
\begin{pmatrix}
  2-2i & i & 3-i \\
   2+2i &-1 & 1+3i \\
 1+5i&  3-i & 3+i
\end{pmatrix},\] then finally we obtain
\[\begin{array}{c}
    x_{11}= {\frac{{{\sum\limits_{\beta \in
J_{2,3} {\left\{ {1} \right\}}} {{\left| \left( {{\rm
{\bf A}}^{2}_{\,.\,1} \left( {\widehat{{\rm
{\bf b}}}_{.1}}^{(1)} \right)} \right){\kern 1pt} {\kern 1pt} _{\beta}
^{\beta} \right|}} } }}{{{\sum\limits_{\beta \in J_{2,3}}
{{\left| {\left( {{\rm {\bf A}}^{ 2} } \right){\kern
1pt} {\kern 1pt} _{\beta} ^{\beta} }  \right|}}} }}}+
 \left ({b_{11}- {\frac{{{\sum\limits_{\beta \in
J_{2,3} {\left\{ {1} \right\}}} {{\left| \left( {{\rm
{\bf A}}^{2}_{\,.\,1} \left( {{\widehat{{\rm
{\bf b}}}}_{.1}}^{(2)} \right)} \right){\kern 1pt} {\kern 1pt} _{\beta}
^{\beta} \right|}} } }}{{{\sum\limits_{\beta \in J_{2,3}}
{{\left| {\left( {{\rm {\bf A}}^{ 2} } \right){\kern
1pt} {\kern 1pt} _{\beta} ^{\beta} }  \right|}}} }}}} \right)t=\\
    \frac{3-3i}{-18i}+ \left(1-\frac{-18i}{-18i}\right)t=\frac{1+i}{6}.
  \end{array}
\]
Similarly,
\[x_{12}=\frac{-3+3i}{-18i}+ \left(i-\frac{9+9i}{-18i}\right)t=\frac{-1-i}{6}+\frac{1+i}{2}t,\,\,x_{13}=0+ \left(1-0\right)t=t,\]
\[x_{21}=\frac{3+3i}{-18i}+ \left(i-\frac{-18}{-18i}\right)t=\frac{-1+i}{6},\]
\[x_{22}=\frac{-3-3i}{-18i}+ \left(0-\frac{-9+9i}{-18i}\right)t=\frac{1-i}{6}+\frac{1+i}{2}t,\,\,x_{23}=0+ \left(1-0\right)t=t,\]
\[\]
\[x_{31}=\frac{-12i}{-18i}+ \left(1-\frac{-18i}{-18i}\right)t=\frac{2}{3},\]
\[x_{32}=\frac{9+3i}{-18i}+ \left(i-\frac{-18}{-18i}\right)t=\frac{-1+3i}{6},\,\,x_{33}=0+ \left(0-0\right)t=0.\]
Then
\[{\rm {\bf X}}= \frac{1}{6}\left(
                                    \begin{array}{ccc}
                                      1+i& -1-i+(3+3i)t& t\\
                                    -1+i& 1-i+(3+3i)t& t\\
                                      4 & -1+3i &0\\
                                    \end{array}
                                  \right)
                                   \]
is  the partial solution of (\ref{eq_ex:X'=AX+B}) .

\end{document}